\documentclass[]{interact}
\usepackage{latexsym}
\usepackage{anyfontsize} 
\usepackage{tikz} 
\usetikzlibrary{cd} 
\usepackage{color}
\usepackage{hyperref}
\usepackage{url}
\usepackage{breakurl}
\newcommand{\bburl}[1]{\textcolor{blue}{\url{#1}}}
\usepackage{float}
\usepackage{subcaption}
\makeatletter
\renewcommand{\fnum@figure}{\figurename\nobreakspace\thefigure}
\makeatother

\makeatletter
\newcommand{\monthyear}[1]{%
  \def\@monthyear{\uppercase{#1}}}
\newcommand{\volnumber}[1]{%
  \def\@volnumber{\uppercase{#1}}}
\makeatother

\theoremstyle{plain}
\numberwithin{equation}{section}
\newtheorem{theorem}{Theorem}[section]

\newtheorem{proposition}[theorem]{Proposition}
\newtheorem{corollary}[theorem]{Corollary}

\theoremstyle{definition}
\newtheorem{definition}[theorem]{Definition}

\theoremstyle{remark}
\newtheorem{remark}[theorem]{Remark}

\numberwithin{table}{section} 
\numberwithin{figure}{section}

\begin{document}

\monthyear{Month Year}
\volnumber{Volume, Number}
\setcounter{page}{1}

\title{The Fibonacci Rectangle Game: Two First-Move Classes and a Triangle-Induced Choice}

\author{
\name{Douglas Larsson Engholm\textsuperscript{b} and Steven J. Miller\textsuperscript{a}}
\affil{
\textsuperscript{a}Department of Mathematics and Statistics, Williams College, Williamstown, MA, USA\\
\textsuperscript{b}Independent Researcher, Sweden
}
}

\maketitle

{\bf Article type}: mathematical outreach 
\bigskip

\begin{abstract}
A square-adjoining rectangle game generates the Fibonacci numbers and the Fibonacci spiral from a simple geometric rule. If one starts from a square, the four possible first moves are all equivalent by rotation. If one starts instead from a non-square rectangle, there are still four geometric placements for the first square, but they split into exactly two equivalence classes: \emph{long-side-first} and \emph{short-side-first}. We show that both classes are governed by the same Fibonacci-type recursion with different initial conditions, and that in both cases the successive aspect ratios converge to the golden ratio $\varphi$. We then add a brief geometric remark: the Hypotenuse--Axis Intercept (HAI) construction from a right triangle produces a natural ordered seed whose outward and inward branches determine precisely those two first-move classes.
\end{abstract}


\section{Introduction}
The Fibonacci spiral is one of the best-known visual introductions to the Fibonacci numbers. A particularly effective classroom version is the square-adjoining rectangle game: place square tiles one at a time on a table, never overlap them, and require that after each move the union is still a rectangle. The second named author has played this game with students of all ages, even as young as kindergarten. Using a set of tiles, students quickly discover that in the original game one cannot proceed beyond a single tile, and that adding a second $1 \times 1$ square produces the sequence of side lengths $2,3,5,8,13$. After writing these on the board, students often predict subsequent terms and uncover the recursive rule themselves. These explorations also lead naturally to interesting relations among Fibonacci numbers; see \cite{CheighFiboDigits}. For outreach-oriented presentations and classroom materials, see \cite{MillerVideo}. For general background on Fibonacci and Lucas numbers, see \cite{Koshy}.

The standard version begins with square tiles of side lengths $1,2,3,4,\dots$, one of each. If one places the $1\times1$ square first, there is no second $1\times1$ square available to keep the shape rectangular, and the game stops immediately. The cheapest way to make the game nontrivial is therefore to add one extra $1\times1$ square. Then the two $1\times1$ squares form a $1\times2$ rectangle, the $2\times2$ square fits next, then the $3\times3$ square, and so on. See Figure~\ref{fig:forced-steps}.

The point of this note is that there is a clean way to generalize the \emph{first move}. Starting from a square, four placements are possible but all are equivalent. Starting from a non-square rectangle, four placements are possible, but now they fall into two genuinely different classes. Those two classes are the main story here. The right-triangle construction enters only at the end: it is a short geometric mechanism that determines the initial class. For further details on the HAI construction, see the extended companion preprint \cite{EngholmHAI}.

\section{The original rectangle game}

As remarked in the introduction, if one begins with a single copy of each $n\times n$ square and places the $1\times1$ square first, the game is trivial: there is no second $1\times1$ square available to maintain a rectangular shape. The smallest change that makes the game nontrivial is therefore to add a second $1\times1$ square.

\begin{figure}[H]
  \centering
  \begin{subfigure}[t]{0.32\linewidth}
    \centering\vspace{0pt}
   \includegraphics[width=\linewidth,height=4.8cm,keepaspectratio]{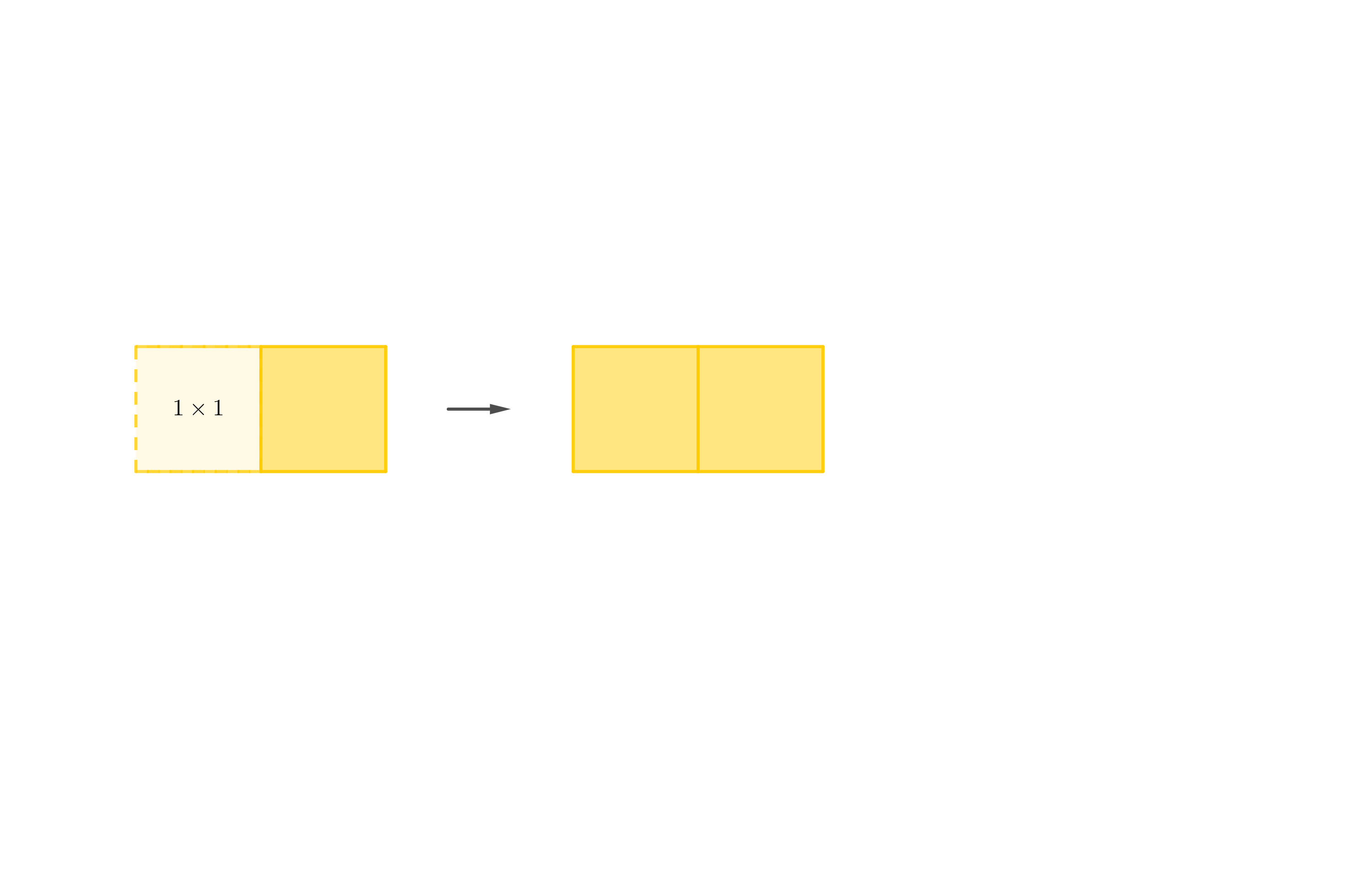}
    \caption{$1\times2$ seed.}
  \end{subfigure}\hfill
  \begin{subfigure}[t]{0.32\linewidth}
    \centering\vspace{0pt}
    \includegraphics[width=\linewidth,height=4.8cm,keepaspectratio]{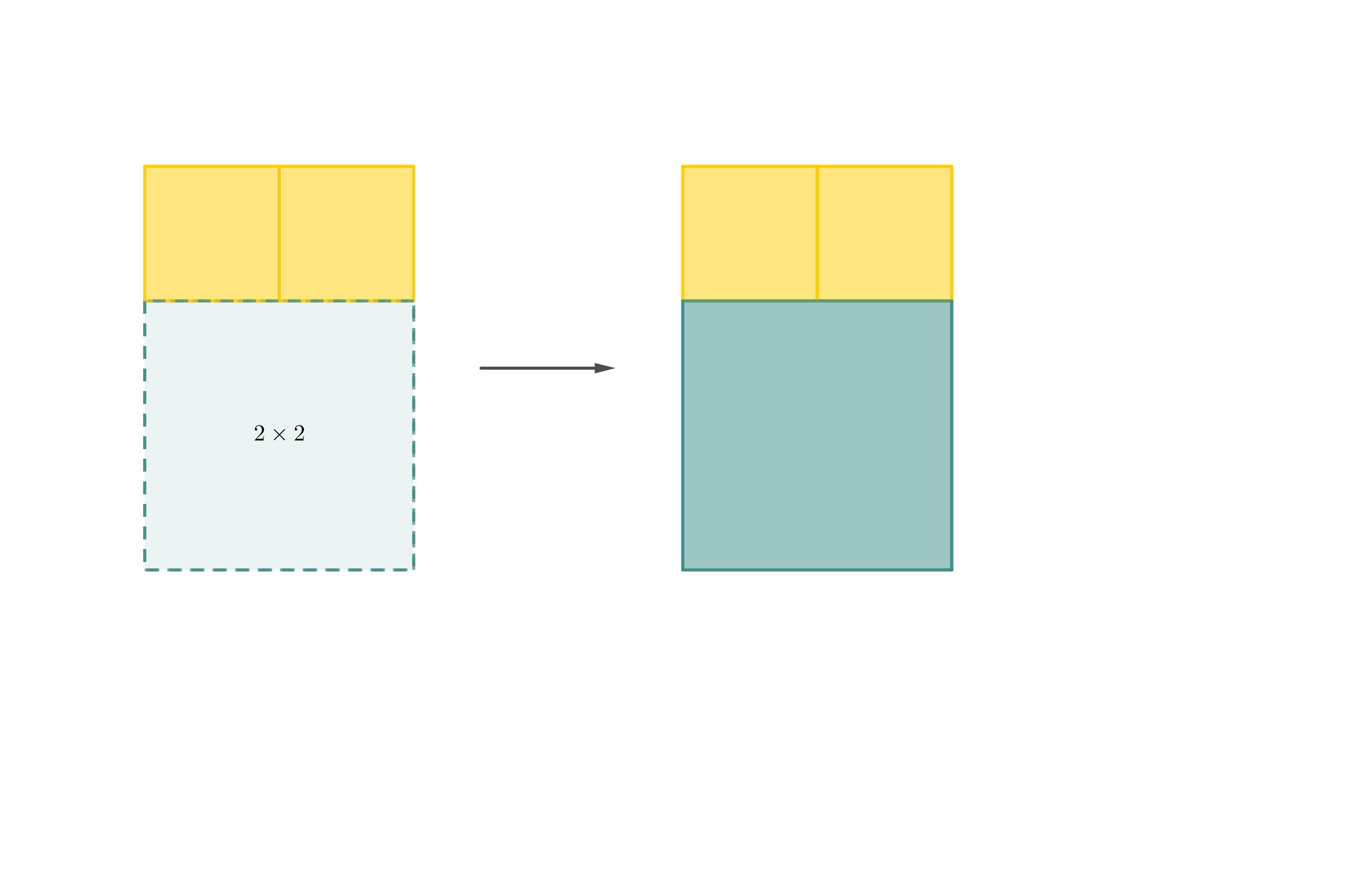}
    \caption{$2\times3$.}
  \end{subfigure}\hfill
  \begin{subfigure}[t]{0.32\linewidth}
    \centering\vspace{0pt}
    \includegraphics[width=\linewidth,height=4.8cm,keepaspectratio]{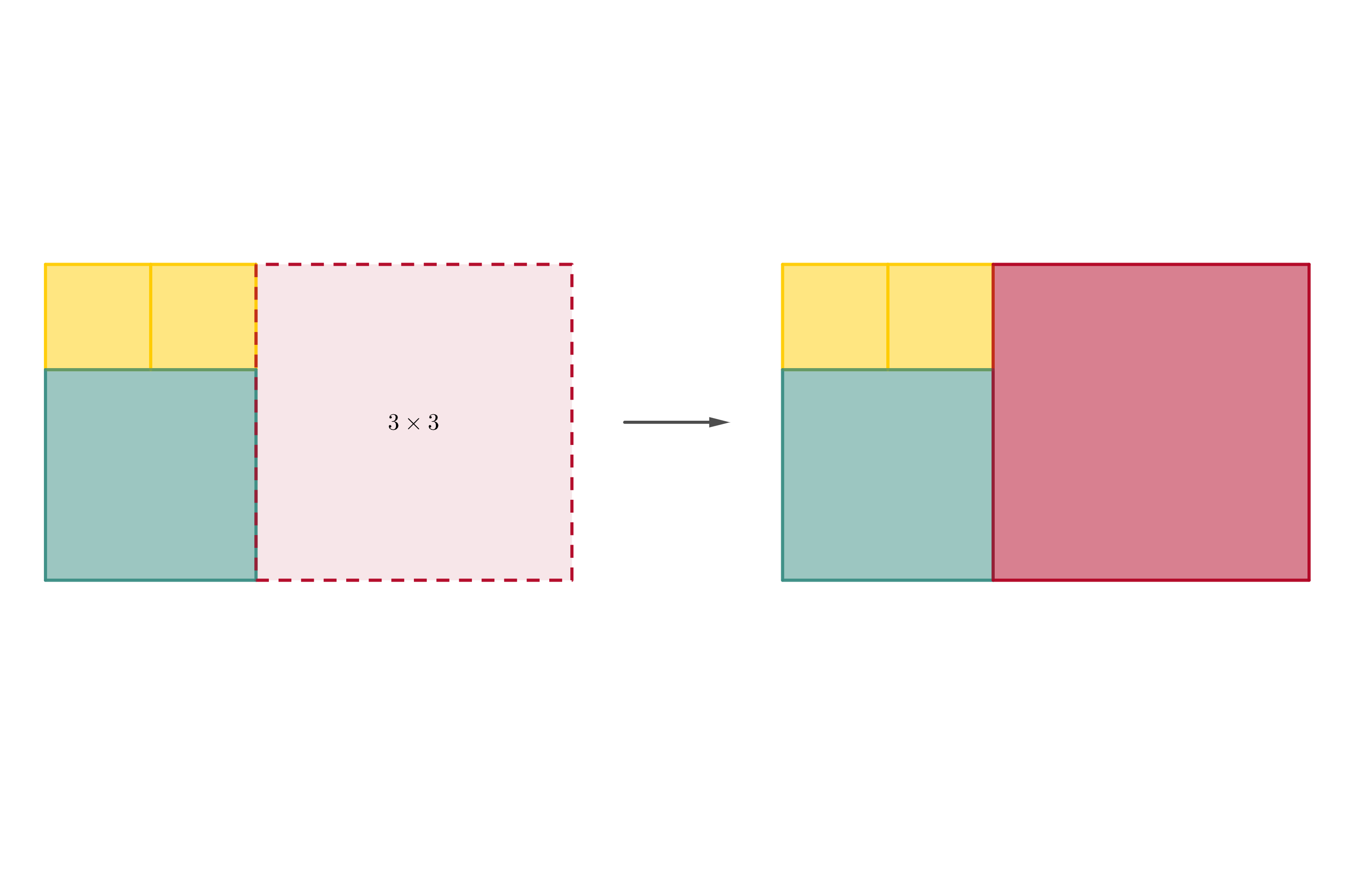}
    \caption{$3\times5$.}
  \end{subfigure}

  \vspace{2mm}

  \begin{subfigure}[t]{0.49\linewidth}
    \centering\vspace{0pt}
  \includegraphics[width=\linewidth,height=4.8cm,keepaspectratio]{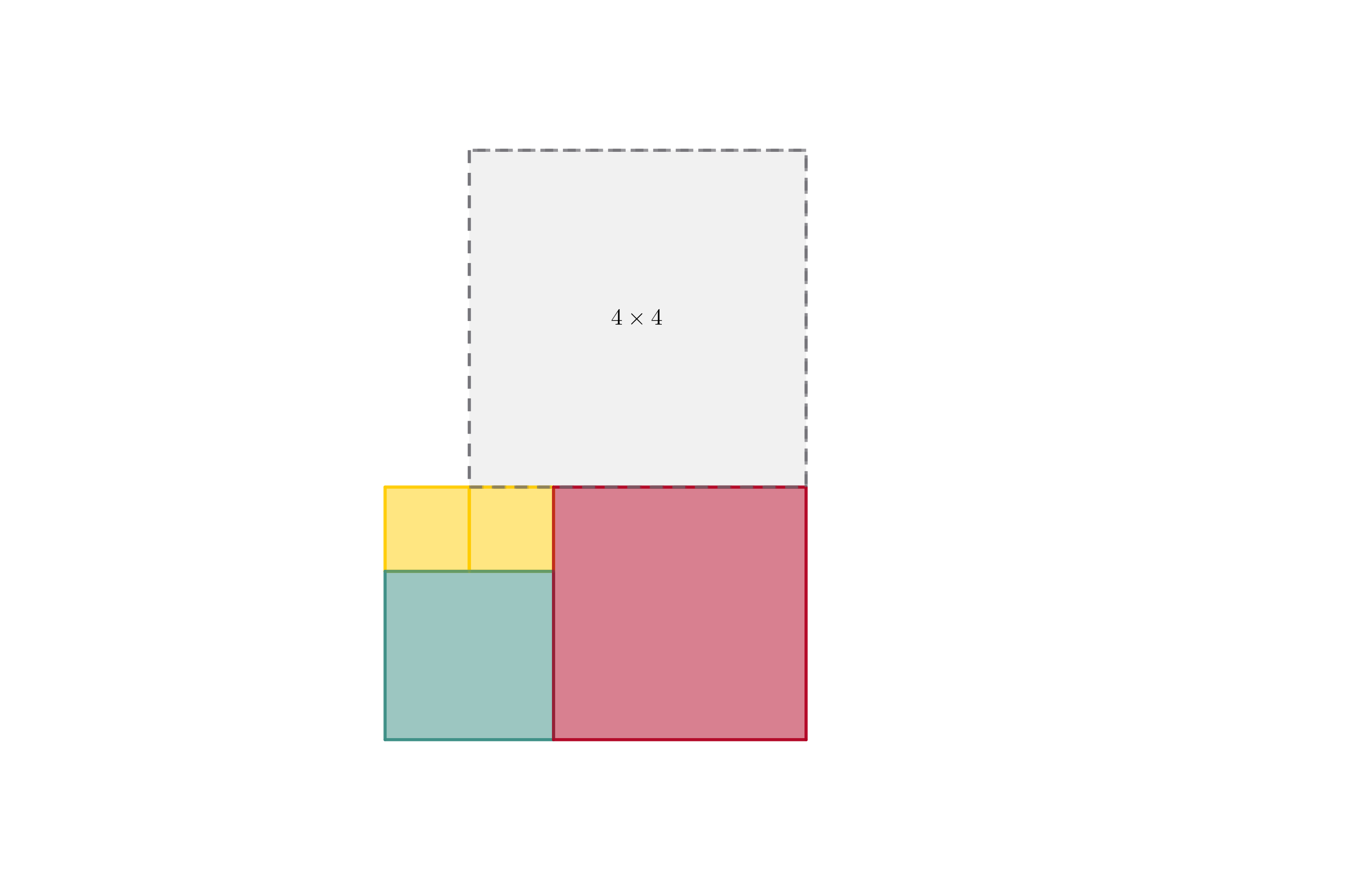}
    \caption{$4\times4$ is impossible here.}
  \end{subfigure}\hfill
  \begin{subfigure}[t]{0.49\linewidth}
    \centering\vspace{0pt}
    \includegraphics[width=\linewidth,height=4.8cm,keepaspectratio]{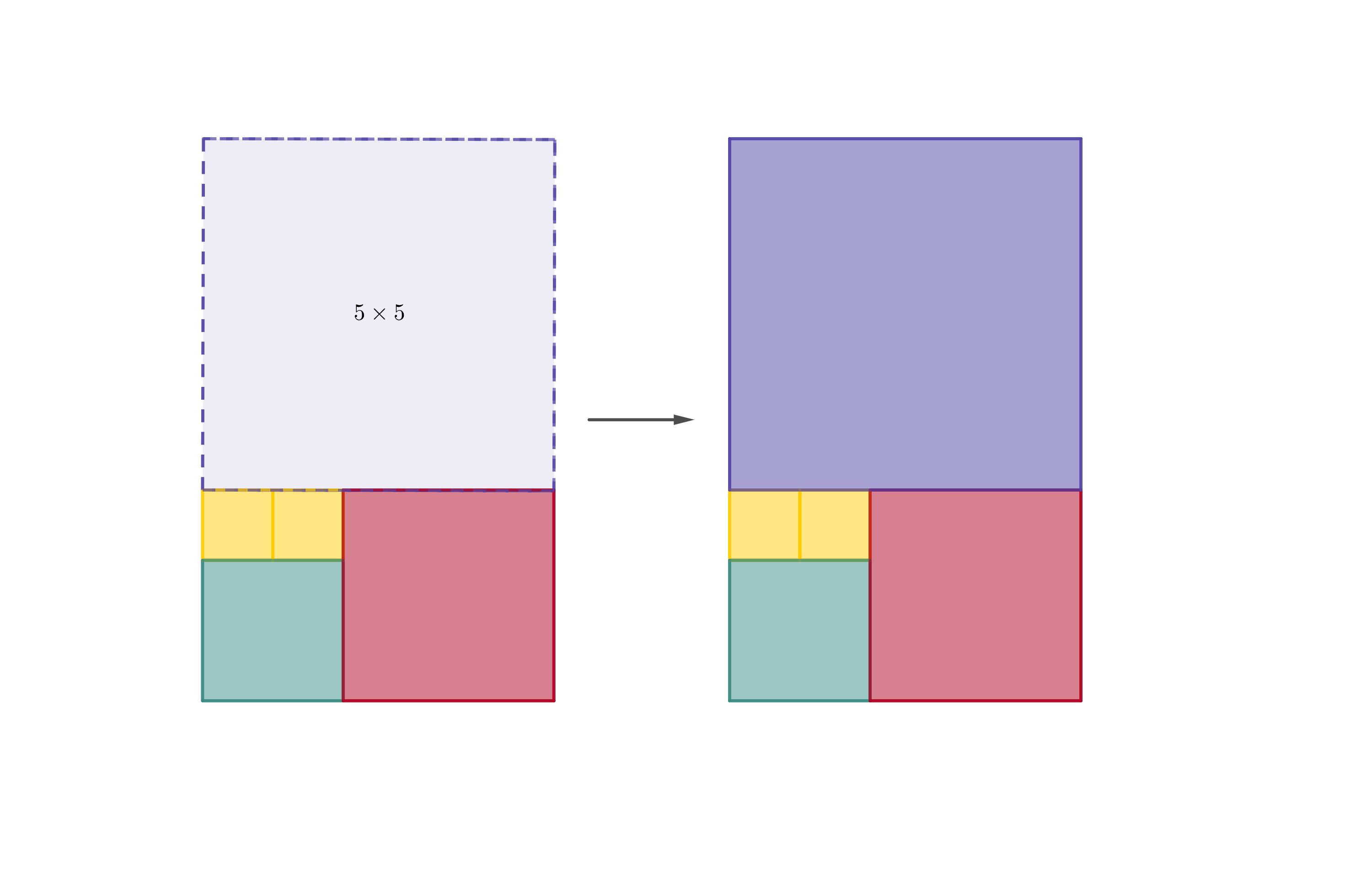}
    \caption{The next legal move is $5\times5$, giving $5\times8$.}
  \end{subfigure}

  \caption{Early forced steps in the rectangle game once a second $1\times1$ square is added.}
  \label{fig:forced-steps}
\end{figure}

\begin{definition}
At stage $n$, let the bounding rectangle have side lengths
\[
R_n \ :=\ (s_n,\ell_n), \qquad 0 \ <\ s_n \ \le\ \ell_n,
\]
where $s_n$ is the short side and $\ell_n$ is the long side.
\end{definition}

\begin{proposition}[Rectangle update]\label{prop:update}
If $R_n = (s_n,\ell_n)$ and one adjoins an $\ell_n\times\ell_n$ square along a side of length $\ell_n$, then
\[
R_{n+1} \ =\ (\ell_n,\,s_n+\ell_n).
\]
Equivalently,
\[
s_{n+1} \ =\ \ell_n,\qquad \ell_{n+1} \ =\ s_n+\ell_n.
\]
\end{proposition}

\begin{proof}
The side of length $\ell_n$ remains unchanged, while the orthogonal side increases from $s_n$ to $s_n+\ell_n$, which is now the longer side.
\end{proof}

\begin{remark}\label{rem:no-return}
Once a square size cannot be used, it can never become usable later. Indeed, Proposition~\ref{prop:update} shows that the long side strictly increases at each step, while the new short side is the previous long side. Hence a missed size can never reappear as the required long side. In particular, once the $4\times4$ tile is impossible in Figure~\ref{fig:forced-steps}(d), it is gone forever.
\end{remark}

Starting from the $1\times2$ seed, the rectangles therefore evolve as
\[
(1,2),\ (2,3),\ (3,5),\ (5,8),\dots,
\]
and the side lengths follow the Fibonacci recursion.

\begin{definition}[Fibonacci numbers]
The Fibonacci numbers are defined by $F_0=0$, $F_1=1$, and $F_{n+1}=F_n+F_{n-1}$ for $n\ge 1$.
\end{definition}

\begin{remark}
With the convention that the two $1\times1$ squares have side lengths $F_1=F_2=1$, the rectangle present after the square of side $F_n$ has been added has dimensions $(F_n,F_{n+1})$.
\end{remark}

The familiar area identity follows immediately.

\begin{theorem}[Sum of squares identity]
For every $n\ge 1$,
\[
F_1^2+F_2^2+\cdots+F_n^2 \ =\ F_nF_{n+1}.
\]
\end{theorem}

\begin{proof}
The identity is classical; see for example \cite{BenjaminQuinn} for a proof; for completeness we provide one below. At stage $n$, the bounding rectangle has dimensions $(F_n,F_{n+1})$, hence area $F_nF_{n+1}$. The same region is tiled by squares of sides $F_1,\dots,F_n$, so its area is $\sum_{k=1}^n F_k^2$. Equating the two expressions gives the identity.
\end{proof}

\begin{figure}[H]
  \centering
  \includegraphics[width=0.34\linewidth]{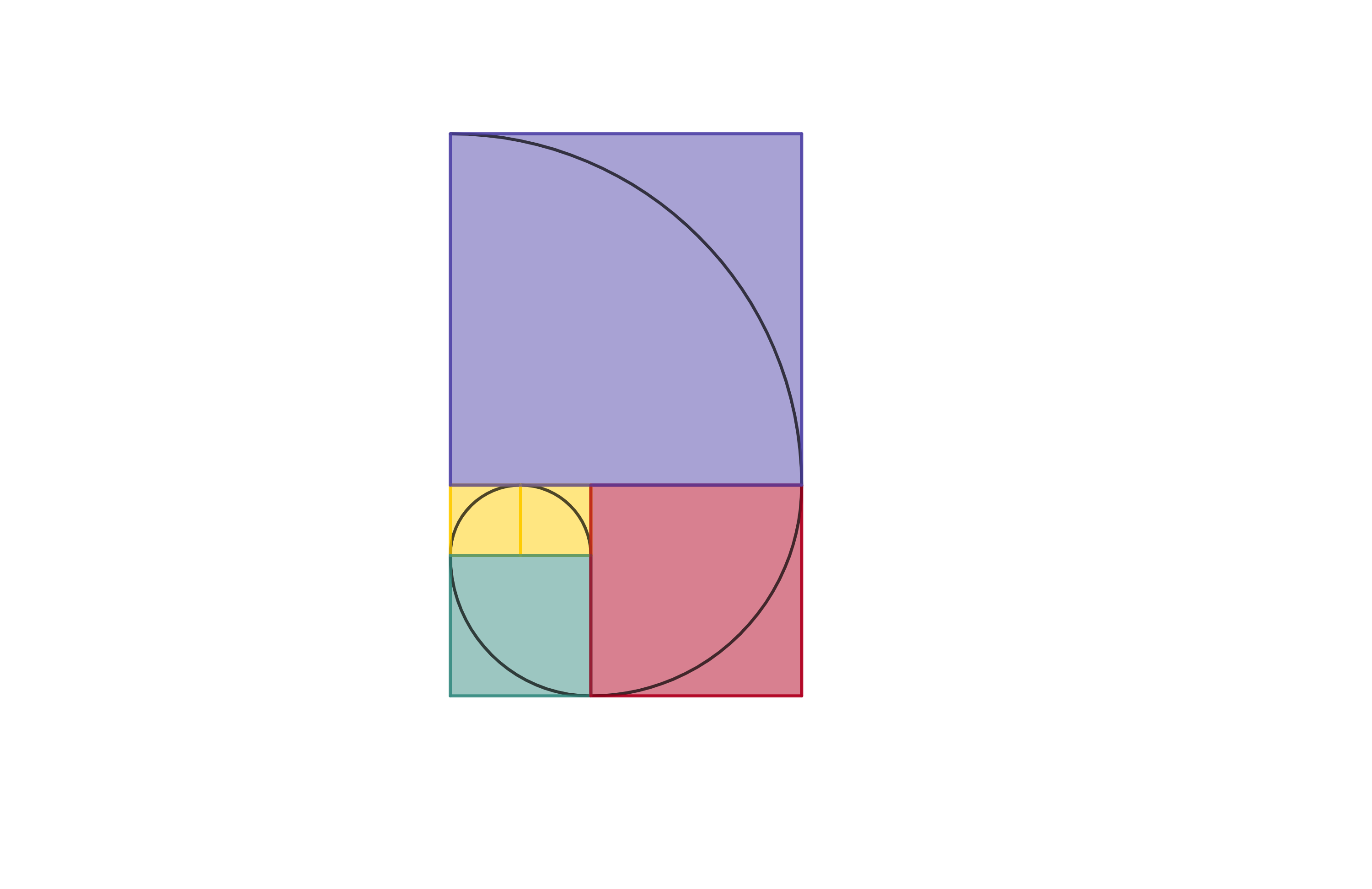}
  \caption{Quarter-circles drawn inside the successive squares produce the familiar Fibonacci spiral.}
  \label{fig:classical-spiral}
\end{figure}

\section{Square seed versus rectangle seed}
If one starts from a square, there are four places to attach the first square, but all four are equivalent by rotation. In that sense there is only one first-move behaviour.

If instead one starts from a non-square rectangle, the situation changes. There are still four geometric placements for the first square, but now two of them lie on long sides and two lie on short sides. These are not equivalent.

\begin{proposition}[Four choices, two equivalence classes]\label{prop:four-two}
Let the starting rectangle have side lengths $S_1$ and $S_2$ with
\[
0 \ <\ S_1 \ <\ S_2.
\]
The four possible first square placements split into exactly two equivalence classes up to rotation:
\begin{enumerate}
\renewcommand{\labelenumi}{(\roman{enumi})}
\item \textbf{Long-side-first:} attach the square of side length $S_2$,
\item \textbf{Short-side-first:} attach the square of side length $S_1$.
\end{enumerate}
The two placements within each class are related by a half-turn.
\end{proposition}

\begin{proof}
The two long sides are opposite, and the two short sides are opposite. Hence the two long-side placements are rotationally equivalent, and the two short-side placements are rotationally equivalent. Since $S_1 \neq S_2$, attaching an $S_2\times S_2$ square is not equivalent to attaching an $S_1\times S_1$ square.
\end{proof}

\begin{remark}
If $S_2/S_1$ is a positive integer, then by changing units this is equivalent to starting with a $1 \times (S_2/S_1)$ integer-sided rectangle. In general, of course, $S_2/S_1$ need not be an integer; if we rescale, the same argument applies, except that instead of having one of each integer-sided square we would have one of each rational-sided square.
\end{remark}

This suggests that the right object is not merely an unordered rectangle, but an \emph{ordered seed}.

\begin{definition}[Ordered seed]
An ordered seed is a pair $(S_1,S_2)$ of positive numbers. Geometrically, one starts from an $S_1\times S_2$ rectangle and declares that the first appended square has side length $S_2$.
\end{definition}

For the same underlying non-square rectangle, the two classes in Proposition~\ref{prop:four-two} are represented by the two ordered seeds
\[
(S_1,S_2)\qquad\text{and}\qquad(S_2,S_1).
\]
Equivalently, if
\[
\rho_2 \ :=\ \frac{S_2}{S_1},
\]
then the two classes are exactly
\[
\rho_2 \ >\ 1\qquad\text{and}\qquad \rho_2 \ <\ 1.
\]
The missing case $\rho_2=1$ is exactly the square seed, where the four first moves collapse to a single class by rotation.

The next two figures illustrate these two classes in the normalized case $S_1=1$ and $S_2=N$, which serves as a convenient representative example.

\begin{figure}[H]
  \centering
  \begin{subfigure}[t]{0.32\linewidth}
    \centering\vspace{0pt}
   \parbox[c][4.1cm][c]{\linewidth}{%
  \centering
  \includegraphics[width=\linewidth,height=4.1cm,keepaspectratio]{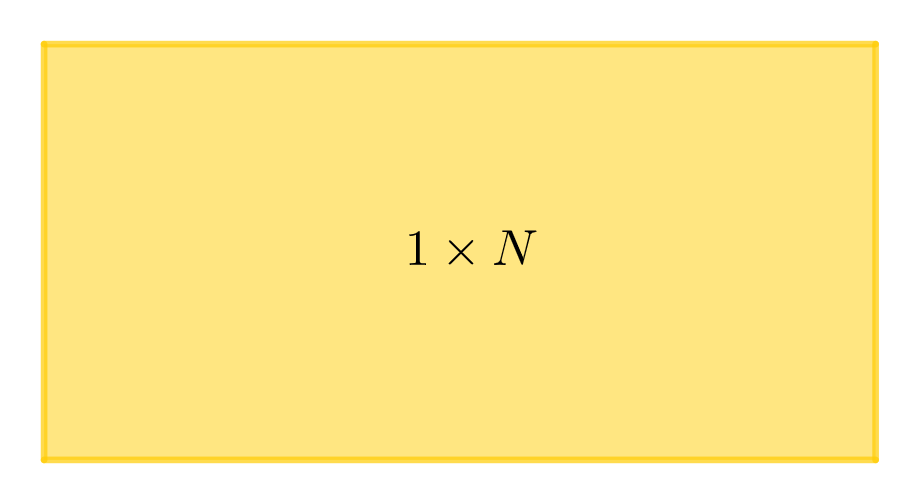}%
}
    \caption{$1\times N$ starting shape.}
  \end{subfigure}\hfill
  \begin{subfigure}[t]{0.32\linewidth}
    \centering\vspace{0pt}
    \parbox[c][4.1cm][c]{\linewidth}{%
  \centering
  \includegraphics[width=\linewidth,height=4.1cm,keepaspectratio]{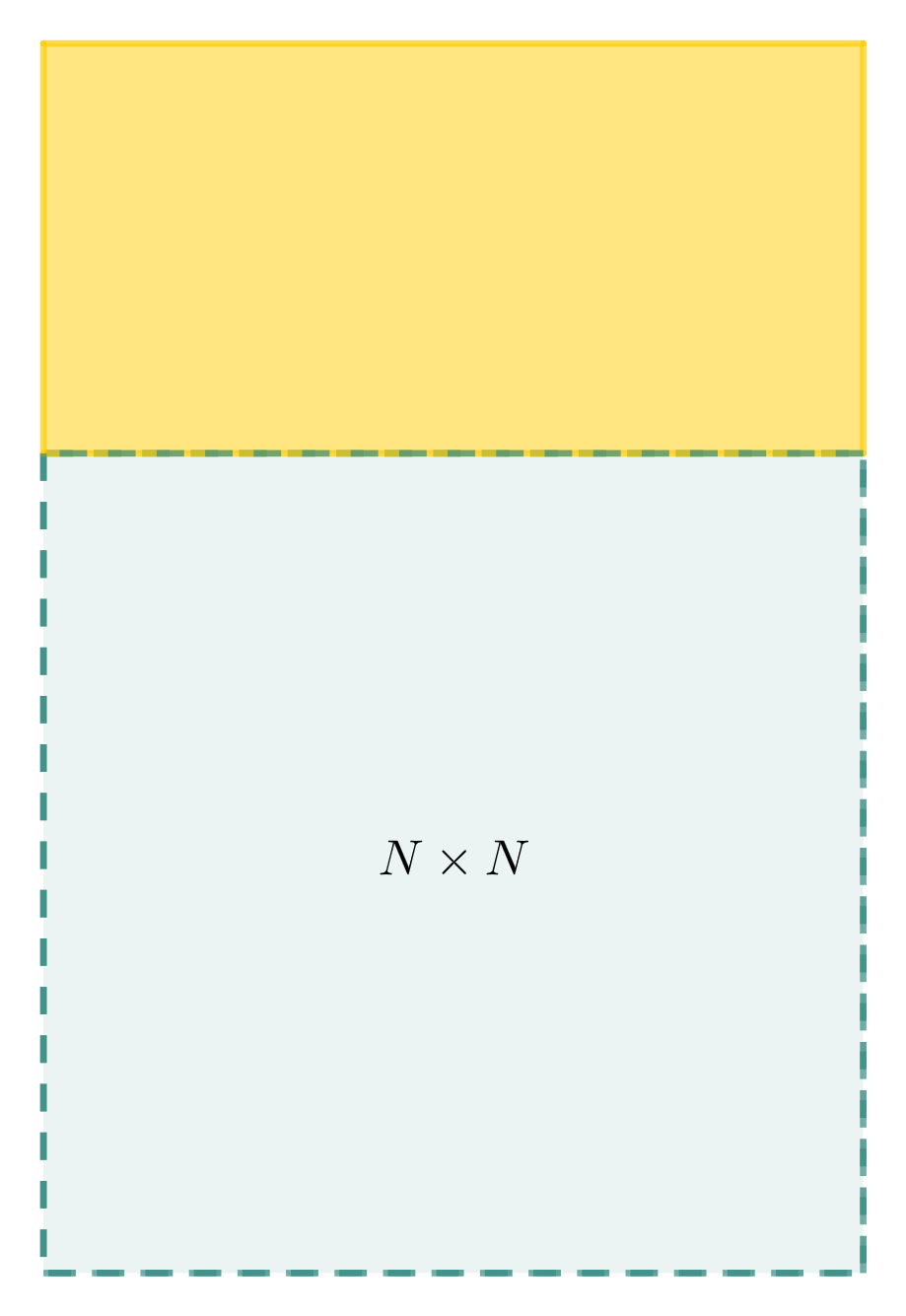}%
}
    \caption{First move: append the $N\times N$ square.}
  \end{subfigure}\hfill
  \begin{subfigure}[t]{0.32\linewidth}
    \centering\vspace{0pt}
   \parbox[c][4.1cm][c]{\linewidth}{%
  \centering
  \includegraphics[width=\linewidth,height=4.1cm,keepaspectratio]{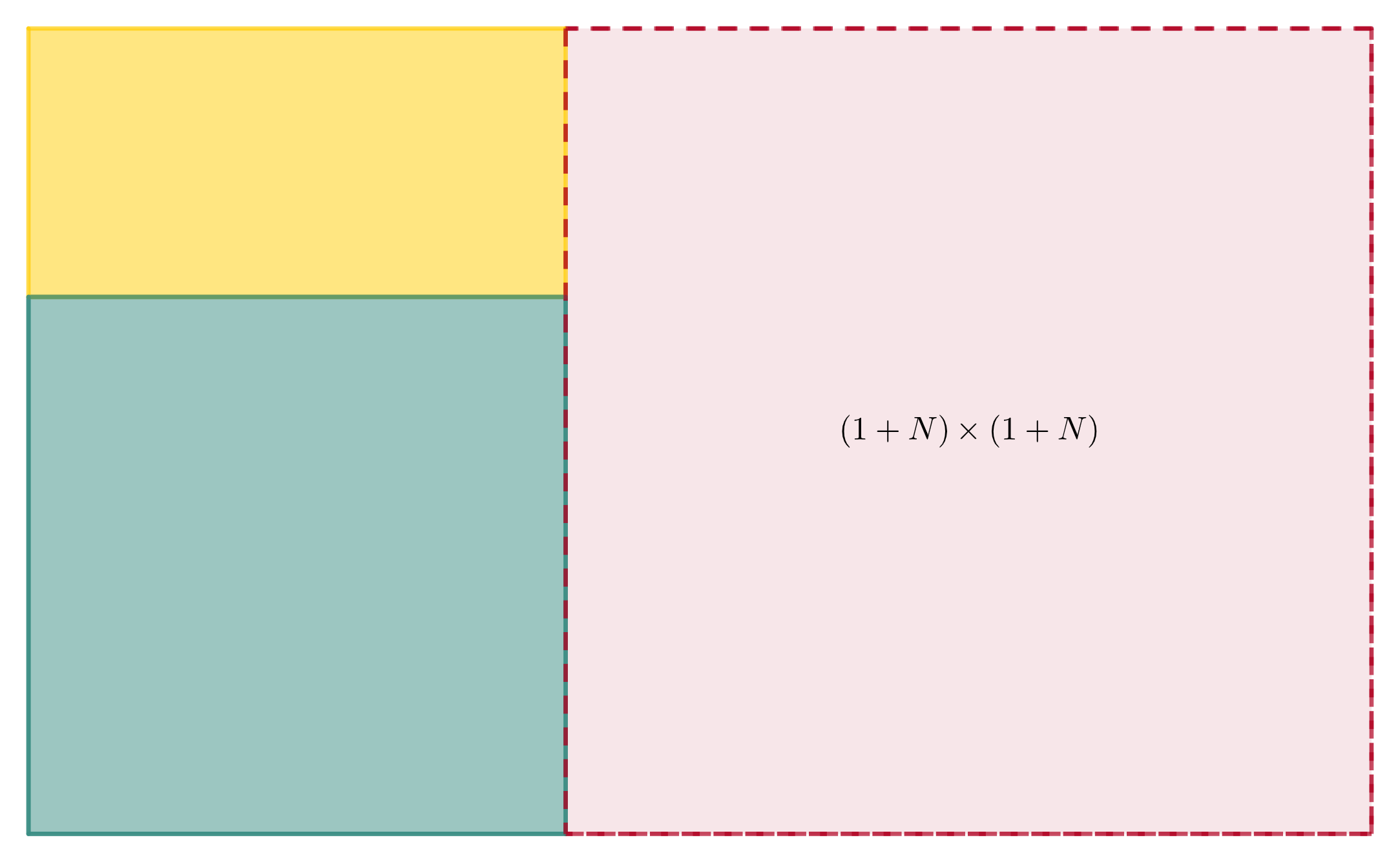}%
}
    \caption{From then on, the same update rule takes over.}
  \end{subfigure}
  \caption{The long-side-first class. The first move is genuinely different, but the subsequent growth is the usual rectangle update.}
  \label{fig:long-first}
\end{figure}

\begin{figure}[H]
  \centering
  \begin{subfigure}[t]{0.32\linewidth}
    \centering\vspace{0pt}
    \parbox[c][4.1cm][c]{\linewidth}{%
  \centering
  \includegraphics[width=\linewidth,height=4.1cm,keepaspectratio]{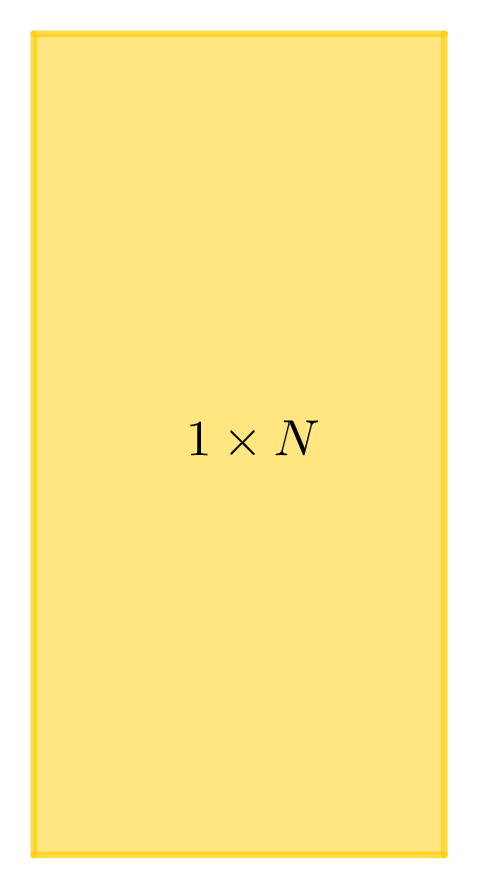}%
}
    \caption{$1\times N$ starting shape, now viewed in the short-side-first orientation.}
  \end{subfigure}\hfill
  \begin{subfigure}[t]{0.32\linewidth}
    \centering\vspace{0pt}
   \parbox[c][4.1cm][c]{\linewidth}{%
  \centering
  \includegraphics[width=\linewidth,height=4.1cm,keepaspectratio]{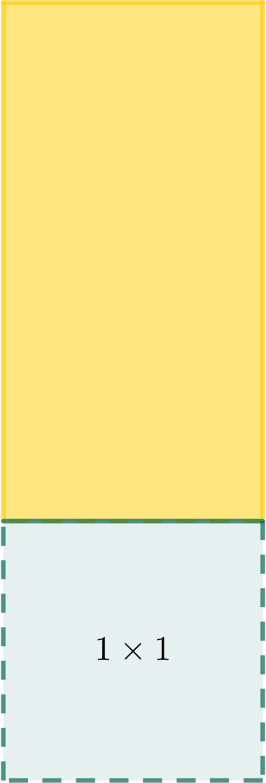}%
}
    \caption{First move: append the $1\times 1$ square.}
  \end{subfigure}\hfill
  \begin{subfigure}[t]{0.32\linewidth}
    \centering\vspace{0pt}
    \parbox[c][4.1cm][c]{\linewidth}{%
  \centering
  \includegraphics[width=\linewidth,height=4.1cm,keepaspectratio]{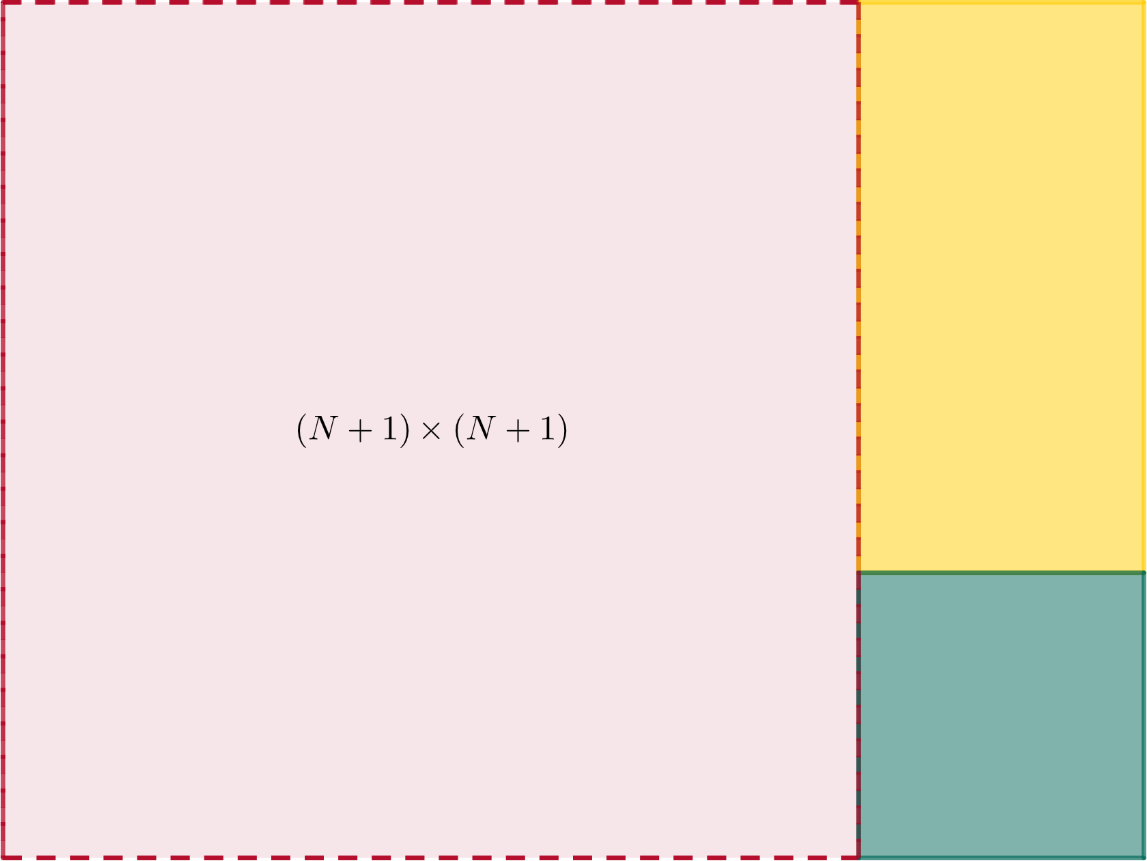}%
}
    \caption{From then on, the same update rule takes over.}
  \end{subfigure}
  \caption{The short-side-first class. The initial move differs, but the later evolution is again Fibonacci-type.}
  \label{fig:short-first}
\end{figure}

Once an ordered seed $(S_1,S_2)$ is fixed, the entire square-adjoining cascade is determined by the same recursion, regardless of the initial class.

\begin{definition}[Span sequence]
Given positive $S_1$ and $S_2$, define $(S_n)_{n\ge 1}$ by
\[
S_{n+1} \ =\ S_n+S_{n-1}\qquad(n \ \ge\ 2).
\]
\end{definition}

\begin{proposition}[Bounding rectangles from an ordered seed]\label{prop:seed-rectangles}
Starting from the ordered seed $(S_1,S_2)$, for each $n\ge 2$, after the square of side $S_n$ has been appended, the bounding rectangle has side lengths $S_n$ and $S_{n+1}$.
\end{proposition}

\begin{proof}
The first move replaces the side $S_1$ by $S_3:=S_1+S_2$ while keeping $S_2$ fixed, so the new rectangle has side lengths $S_2$ and $S_3$. The same argument iterates.
\end{proof}

\begin{proposition}[Closed form]\label{prop:closed}
For every ordered seed,
\[
S_n \ =\ F_{n-2}S_1+F_{n-1}S_2\qquad(n \ \ge\ 2).
\]
\end{proposition}

\begin{proof}
The formula is immediate for $n=2$ and $n=3$. Assume the formula holds for $n$ and $n-1$, where $n\ge 3$. Then
\[
S_{n+1} \ =\ S_n+S_{n-1}
\ =\ (F_{n-2}+F_{n-3})S_1+(F_{n-1}+F_{n-2})S_2
\ =\ F_{n-1}S_1+F_nS_2. \qedhere
\]
\end{proof}

\begin{definition}[Ratio dynamics]
For $n\ge 2$, define
\[
\rho_n \ :=\ \frac{S_n}{S_{n-1}}.
\]
\end{definition}

\begin{proposition}[Universal update]\label{prop:ratio}
For every positive seed,
\[
\rho_{n+1} \ =\ 1+\frac{1}{\rho_n}\qquad(n \ \ge\ 2).
\]
\end{proposition}

\begin{proof}
Using the recursion,
\[
\rho_{n+1} \ =\ \frac{S_{n+1}}{S_n} \ =\ \frac{S_n+S_{n-1}}{S_n} \ =\ 1+\frac{S_{n-1}}{S_n}\ =\ 1+\frac{1}{\rho_n}.\qedhere
\]
\end{proof}

\begin{theorem}[Same limit in both classes]\label{thm:phi}
For every positive ordered seed $(S_1,S_2)$, the ratios $\rho_n$ are defined for $n\ge 2$ and satisfy
\[
\rho_n \ \longrightarrow\ \varphi \ :=\ \frac{1+\sqrt5}{2}.
\]
In particular, both first-move classes converge to the same limiting aspect ratio.
\end{theorem}

\begin{proof}
For $n\ge 3$, Proposition~\ref{prop:closed} gives
\[
\rho_n \ =\ \frac{S_n}{S_{n-1}}
\ =\ \frac{F_{n-2}S_1+F_{n-1}S_2}{F_{n-3}S_1+F_{n-2}S_2}.
\]
Dividing numerator and denominator by $F_{n-2}$ yields
\[
\rho_n \ =\
\frac{S_1+\left(F_{n-1}/F_{n-2}\right)S_2}
{\left(F_{n-3}/F_{n-2}\right)S_1+S_2}.
\]
As is well known, $F_{n-1}/F_{n-2}\to\varphi$ and hence $F_{n-3}/F_{n-2}\to\varphi^{-1}$; see, for example, \cite{Koshy}. Therefore,
\[
\lim_{n\to\infty}\rho_n
\ =\
\frac{S_1+\varphi S_2}{\varphi^{-1}S_1+S_2}
\ =\ \varphi.\qedhere
\]
\end{proof}

\begin{remark}
The difference between the two classes is entirely in the initial condition. If $\rho_2>1$, the process starts long-side-first. If $\rho_2<1$, the first step is short-side-first, but then
\[
\rho_3 \ =\ 1+\frac{1}{\rho_2} \ >\ 2,
\]
so after one step the same growth pattern takes over.
\end{remark}

\section{A triangle can determine the first move}
We now add the brief geometric remark that motivated this note. A right triangle can supply the ordered seed intrinsically, and hence determine one of the two first-move classes. There are four such HAI realizations (two choices of pivot and two branches), corresponding to two outward and two inward constructions; however, for the present discussion it suffices to illustrate one pivot with its outward and inward cases.

Let $\triangle ABC$ be right-angled at $A$. Choose an acute pivot vertex $V$, let $u=|AV|$ be the leg on the landing axis, let $v$ be the other leg, and let $c$ be the hypotenuse. See Figure~\ref{fig:hai-two-branches}.

\begin{figure}[H]
  \centering
  \begin{subfigure}[t]{0.49\linewidth}
    \centering
    \includegraphics[width=\linewidth,height=5.2cm,keepaspectratio]{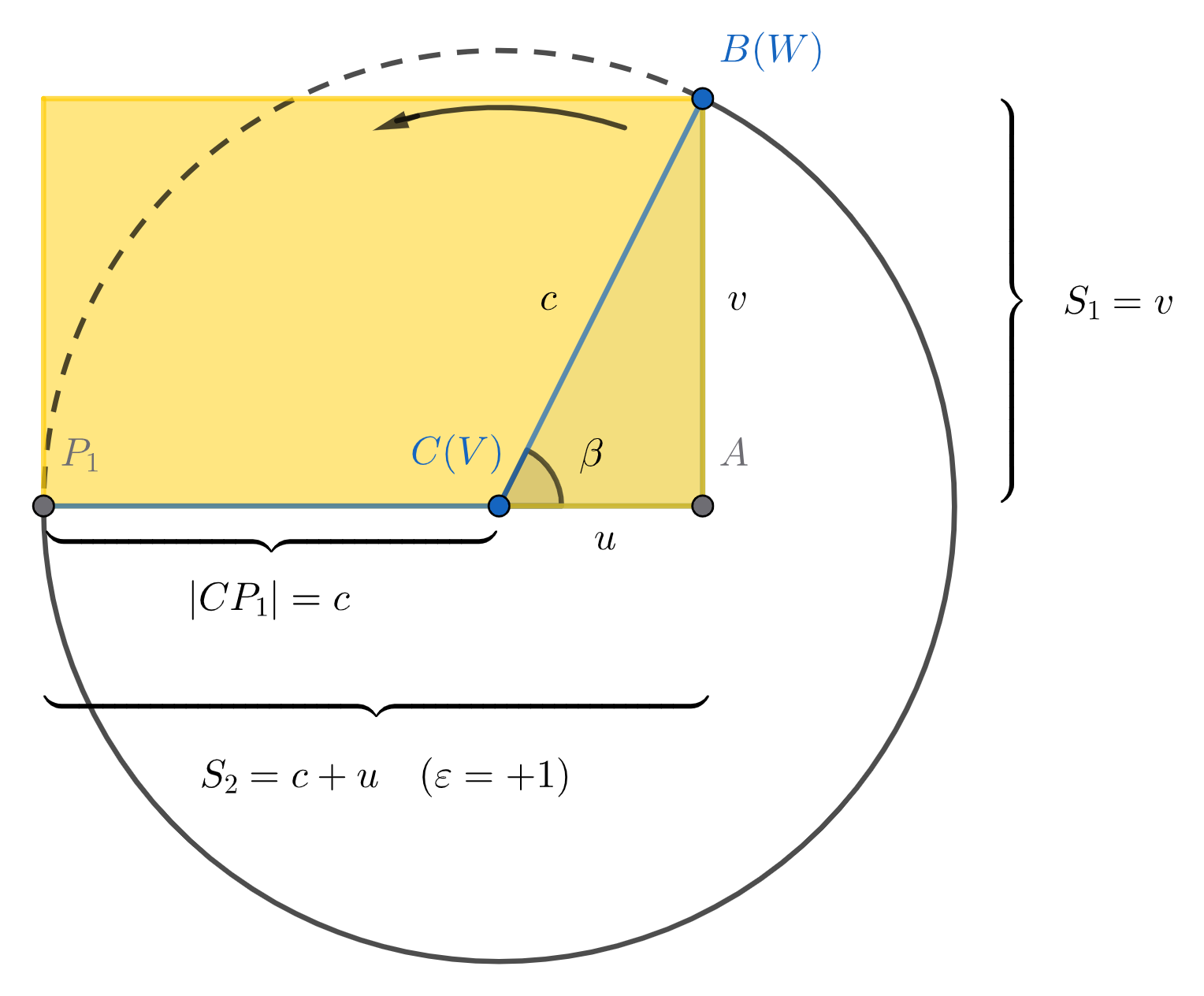}
    \caption{Outward landing: $S_2=c+u$.}
  \end{subfigure}\hfill
  \begin{subfigure}[t]{0.49\linewidth}
    \centering
    \includegraphics[width=\linewidth,height=5.2cm,keepaspectratio]{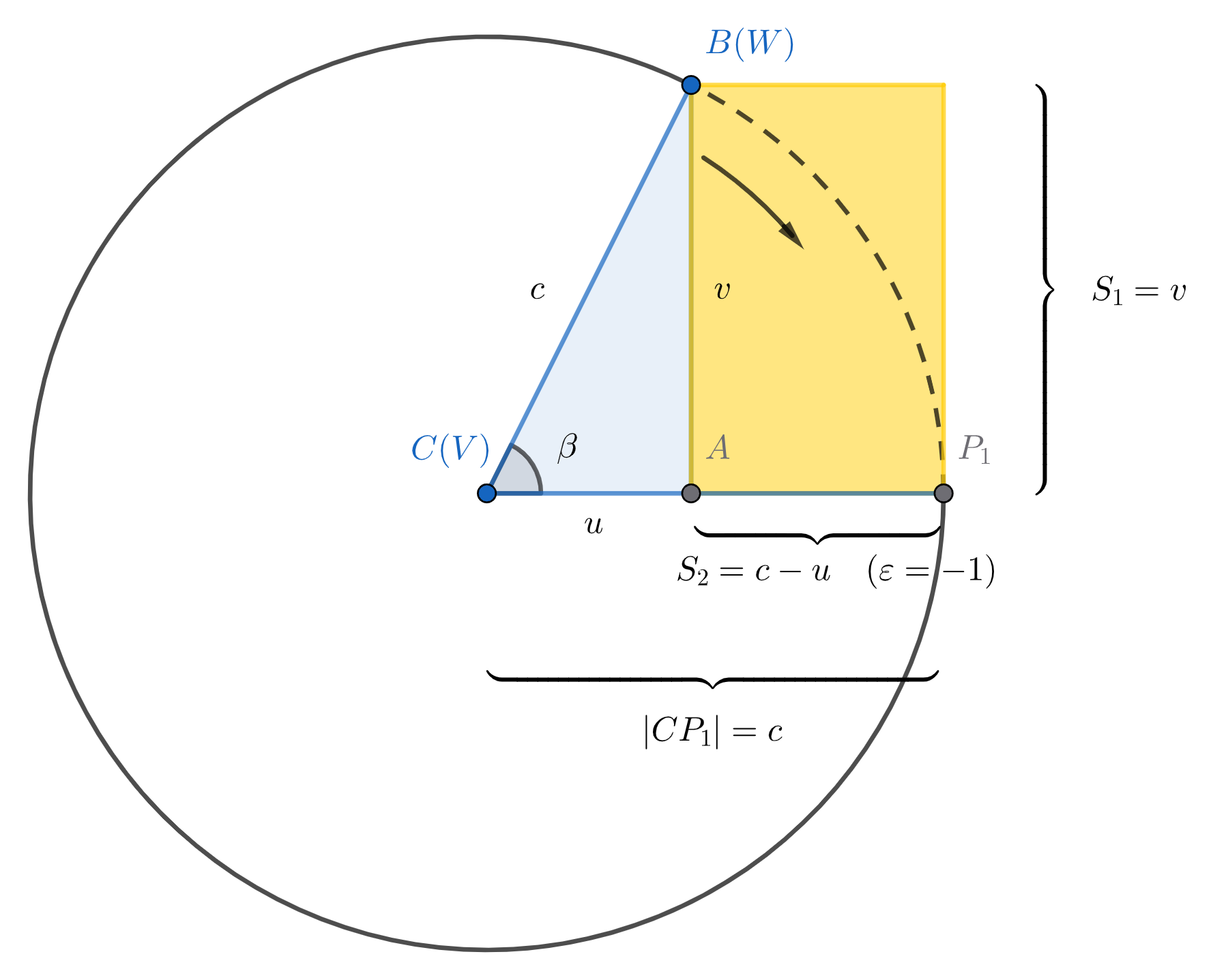}
    \caption{Inward landing: $S_2=c-u$.}
  \end{subfigure}
  \caption{The HAI construction produces the seed $S_1=v$ and determines the first class: outward gives $S_2>S_1$ (long-side-first), while inward gives $S_2<S_1$ (short-side-first).}
  \label{fig:hai-two-branches}
\end{figure}

\begin{proposition}[HAI seed and class selection]\label{prop:hai}
The Hypotenuse--Axis Intercept construction produces the ordered seed
\[
S_1 \ =\ v,\qquad S_2 \ =\ c+\varepsilon u,
\]
where $\varepsilon=+1$ for the outward branch and $\varepsilon=-1$ for the inward branch. Consequently,
\[
S_2 \ >\  S_1 \quad \text{in the outward case},\qquad
S_2 \ <\ S_1 \quad \text{in the inward case}.
\]
Thus the outward and inward branches determine exactly the long-side-first and short-side-first rectangle classes.
\end{proposition}

\begin{proof}
By construction, $S_1$ is the leg orthogonal to the landing axis, so $S_1=v$, while the landing point is at distance $c\pm u$ from $A$. Hence
\[
S_2 \ =\ c+\varepsilon u.
\]
Since $c>u$, the inward quantity $c-u$ is positive. Moreover, $c+u>v$ is immediate, while
\[
c-u \ <\ v \qquad\Longleftrightarrow\qquad c \ <\ u+v,
\]
and the latter holds for every nondegenerate triangle.
\end{proof}

\begin{corollary}[The $1$--$2$--$\sqrt5$ triangle]\label{cor:golden}
In the outward branch with pivot at the acute vertex adjacent to the shorter leg,
\[
\frac{S_2}{S_1} \ =\ \varphi
\qquad\Longleftrightarrow\qquad
v \ =\ 2u.
\]
So the right triangle with side ratio $1:2:\sqrt5$ is the unique similarity class for which the HAI seed is already golden; equivalently, this consists of all triangles of the form $(r,2r,r\sqrt{5})$ with $r>0$.
\end{corollary}

\begin{proof}
In the outward branch,
\[
\frac{S_2}{S_1} \ =\ \frac{c+u}{v} \ =\ \frac{\sqrt{u^2+v^2}+u}{v}.
\]
Writing $\lambda=v/u>0$, this becomes
\[
\frac{\sqrt{1+\lambda^2}+1}{\lambda}.
\]
Setting this equal to $\varphi$ gives $\lambda=2$.
\end{proof}

\begin{remark}
The triangle does not change the rectangle dynamics. It supplies exactly the one datum that the rectangle story leaves open: which first-move class occurs. After that, the update is the same Fibonacci-type cascade studied above.
\end{remark}

\section{Conclusion and Future Work}

There are two main conclusions.

First, the square-adjoining Fibonacci story is naturally a story about \emph{ordered seeds}. Starting from a square, four first moves collapse to one behaviour. Starting from a non-square rectangle, four first moves collapse to two behaviours. Those two behaviours have different initial conditions, but the same recursion and the same limiting ratio $\varphi$.

Second, a right triangle can be used to determine the first move. The HAI construction selects one of the two classes intrinsically: outward means long-side-first, inward means short-side-first. In that sense the triangle does not introduce a new dynamic; it gives a geometric way to choose the initial class.
\medskip

As we have seen, our analysis applies to any ordered seed $(S_1,S_2)$ of positive side lengths. No divisibility assumption is required. The evolution depends only on the initial ratio $\rho_2 = S_2/S_1$. In particular, every non-square rectangle yields the same two classes, and the limiting ratio is $\varphi$.

This raises a natural further question concerning the corresponding area identity. For the classical Fibonacci seed $(1,1)$ one has the well-known formula
\[
\sum_{k=1}^n F_k^2 \ =\ F_nF_{n+1}.
\] 
It is thus natural to ask whether an analogous identity holds for a general ordered seed $(S_1,S_2)$, and how it depends on both the initial values and the initial choice of where to append the next square.

\section*{Acknowledgements}
The first named author used AI-based tools for language editing and assistance with expressing mathematical ideas.

\section*{Disclosure statement}
No conflict of interest has been reported by the authors.

\medskip
\noindent MSC2020: 11B39, 33C05


\begin{thebibliography}{99}

\bibitem{BenjaminQuinn}
Benjamin, A. T., Quinn, J. J. (2003).
\emph{Proofs That Really Count: The Art of Combinatorial Proof}.
Dolciani Mathematical Expositions, Vol.~27.
Washington, DC: Mathematical Association of America.

\bibitem{CheighFiboDigits}
Cheigh, J., Dantas e Moura, G. Z., Duke, J. L., Mauro, A., McDonald, Z., Mello, A., Miller, K., Miller, S. J., Velazquez Iannuzzelli, S. (2025).
An invitation to Fibonacci digits.
\emph{Fibonacci Quart.} 63(4): 702--710.
\href{https://doi.org/10.1080/00150517.2024.2442595}{doi.org/10.1080/00150517.2024.2442595}

\bibitem{Koshy}
Koshy, T. (2001).
\emph{Fibonacci and Lucas Numbers with Applications}.
New York, NY: Wiley-Interscience.

\bibitem{EngholmHAI}
Larsson Engholm, D. (2026).
From right triangles to Fibonacci-type spirals via the Hypotenuse--Axis Intercept (HAI): the characteristic identity, M\"obius maps, Lucas traces, and the golden 2:1 triangle.
Zenodo, Version 4.0.
\href{https://doi.org/10.5281/zenodo.18486177}{doi.org/10.5281/zenodo.18486177}

\bibitem{MillerVideo}
Miller, S. J. (2022).
The Fibonacci sequence and math outreach.
Video presentation, 20th International Fibonacci Conference, Sarajevo, July 29, 2022.
\href{https://youtu.be/vZOttB4ksXw}{youtu.be/vZOttB4ksXw}

\end{thebibliography}
\end{document}